\documentclass{article}
\usepackage{amsmath,amssymb,amscd}
\usepackage{amsthm}
\newtheorem{thm}{\indent \sc Theorem}[section]
\newtheorem{prop}{\indent \sc Proposition}[section]
\newtheorem{cor}{\indent \sc Corollary}
\newtheorem{defi}{\indent \sc Definition}
\newtheorem{lem}{\indent \sc Lemma}[section]

\newtheorem{rem}{Remark}
\newcommand{\spec}{\operatorname{Spec}}
\newtheorem{pr}{Property}
\newtheorem{prm}{Problem}

\begin{document}
\author{Makoto Sakagaito}
\title{On Problems about a generalization of the Brauer group}
\date{}
\maketitle
\begin{center}
School of Mathematics, Harish-Chandra Research Institute \\
Allahabad (India) 211019 \\
\textit{E-mail address}: sakagaito43@gmail.com
\end{center}
\begin{abstract}
In this paper, we define a generalization of the Brauer groups by using Bloch's cycle complex
on etale site. 
We prove the Gersten conjecture of the generalized Brauer group on some cases. 
As an application we prove the Gersten conjecture of the logarithmic Hodge-Witt cohomology
for a two dimensional regular local ring which is smooth over the spectrum of some 
discrete
valuation ring of characteristic $p>0$.
Moreover we consider
a generalization of Artin's theorem on Brauer groups.
\end{abstract}
\section{Introduction}

Let $X$ be a regular scheme and 
$\operatorname{Br}(X)$ be the cohomological Brauer group 
\begin{math}
\operatorname{H}^{2}(X_{et}, \mathbb{G}_{m}).
\end{math}

Then $\operatorname{Br}(X)$ satisfies the following properties:

\begin{pr}\label{per}
\begin{description}\upshape
\item[(i)](cf. Remark \ref{pic}) Let $X$ be a regular local ring.

\item[(i-a)] Let $l$ be an integer which is invertible in $X$. Then
\begin{equation*}
\operatorname{Br}(X)_{l}=
\operatorname{Ker}\left(\operatorname{Br}(X)\xrightarrow{\times l}\operatorname{Br}(X)\right)=
\operatorname{H}^{2}_{et}(X, \mu_{l})
\end{equation*}
where $\mu_{l}$ is the sheaf of $l$-th roots of unity.

\item[(i-b)] Let $\operatorname{char}(X)=p>0$. Then
\begin{equation*}
\operatorname{Br}(X)_{p^{r}}=
\operatorname{H}^{1}_{et}(X, \nu^{n}_{r})
\end{equation*}
where $\nu^{n}_{r}$ is the logarithmic de Rham-Witt sheaf $\operatorname{W\Omega}^{n}_{r, \log}$.
\item[(ii)] (Purity) (\cite{Ga}, \cite{Sh})
Let $X$ be an equi-characteristic regular local ring and $Z\subset X$
a regular closed subscheme of codimension $\geq 2$. 
Then
\begin{equation*}
\operatorname{Br}(X)=\operatorname{Br}(X\setminus Z).
\end{equation*}
\item[(iii)](cf.\cite[p.93, Proposition (2.1)]{Gr}, \cite[Proposition 7.14]{Sak}) 
Let $X$ be a Noether regular integral scheme and $K$ the ring of rational functions on
$X$. Let $X^{(i)}$ the set of points $x$ of codimension $i$ (that is, such that 
$\operatorname{dim}\mathcal{O}_{X, x}=i$) and 
$X_{(i)}$ the set of points $x$
of dimension $i$ (that is such that $\operatorname{dim}\bar{\{x\}}=i$). Let $\kappa(x)$ be the residue field of 
$x\in X$ and 
$K_{\bar{x}}$ the field of
fractions of a strictly Henselian $\mathcal{O}_{X, \bar{x}}$.
Then there is an exact sequence
\begin{equation*}
0\to \operatorname{Br}(X)\to \operatorname{Ker}\left(\operatorname{Br}(K)\xrightarrow{\operatorname{Res}}
\prod_{x\in X_{(0)}}\operatorname{Br}(K_{\bar{x}})\right)
\to\bigoplus_{x\in X^{(1)}}\operatorname{H}^{1}\left(\kappa(x), \mathbb{Q}/\mathbb{Z}\right).
\end{equation*}
\end{description}
\end{pr}
\vspace{0.5mm}
We can define a group $\operatorname{H}^{i}_{\operatorname{B}}(X)$ 
for an equi-dimensional scheme $X$ by using etale motivic cohomology (see $\S 3$).

\vspace{2.0mm}

The group $\operatorname{H}^{i}_{\operatorname{B}}(X)$ satisfies 
\begin{equation*}
\operatorname{H}^{1}_{\operatorname{B}}(X)=\operatorname{H}^{1}(X_{et}, \mathbb{Q}/\mathbb{Z})
~~\textrm{and}~~
\operatorname{H}^{2}_{\mathrm{B}}(X)=\operatorname{Br}(X)
\end{equation*}
for an essentially smooth scheme over a Dedekind ring $X$. 

\vspace{2.0mm}

We can consider the following problems:
\begin{prm}\upshape
\label{prob}
\begin{description}
\item[(i)]  Let $X$ be a regular local ring. 

\item[(i-a)] Let $l$ be an integer which is prime to $\operatorname{char}(X)$. Then
\begin{equation*}
\operatorname{H}^{i}_{\mathrm{B}}(X)_{l}=\operatorname{H}^{i}_{et}(X, \mu^{\otimes (i-1)}_{l}).
\end{equation*}
\item[(i-b)] Let $\operatorname{char}(X)=p>0$. Then
\begin{equation*}
\operatorname{H}^{i}_{\mathrm{B}}(X)_{p^{r}}=\operatorname{H}^{1}_{et}(X, \nu^{i-1}_{r}).
\end{equation*}
\item[(ii)] Let $X$ be a regular local ring and $Z\subset X$ a regular closed subscheme of codimension $\geq 2$.
Then 
\begin{equation*}
\operatorname{H}^{n}_{\mathrm{B}}(X)
=\operatorname{H}^{n}_{\mathrm{B}}(X \setminus Z).
\end{equation*}
\item[(iii)] Let notations be same as Property \ref{per} (iii). Then the sequence
\begin{equation*}
0\to\operatorname{H}_{\operatorname{B}}^{i}(X)\to
\operatorname{Ker}\left(\operatorname{H}_{\operatorname{B}}^{i}(K)\xrightarrow{\operatorname{Res}}\prod_{x\in X_{(0)}}
\operatorname{H}_{\operatorname{B}}^{i}(K_{\bar{x}})\right)
\to \bigoplus_{x\in X^{(1)}}\operatorname{H}_{\operatorname{B}}^{i-1}(\kappa(x))
\end{equation*}
is exact.
\end{description}
\end{prm}

We prove the following results in this paper.
\begin{thm}\upshape
\label{Res1}
\begin{description}
\item[(i)] (Proposition \ref{leD}) Let $X$ be a local ring at a point $x$ of an essentially smooth
scheme over a Dedekind ring.

Then Problem \ref{prob} (i) is true for $X$.
\item[(ii)] (Proposition \ref{supp}) Let $X$ be a regular integral scheme with $\operatorname{dim}(X)=1$. 

Then Problem \ref{prob} (iii) is true for $X$.
\item[(iii)] (Theorem \ref{CDVRS}) Let $\mathcal{O}_{K}$ be a Henselian discrete valuation ring.
Let $\mathfrak{X}$ be an integral regular scheme and proper and flat over $\spec\mathcal{O}_{K}$ with 
$\operatorname{dim}(\mathfrak{X})=2$. 

Then Problem \ref{prob} (iii) 
is true for $\mathfrak{X}$.
\end{description}
\end{thm}
\begin{thm}\upshape
Let $B$ be a regular local ring with $\operatorname{dim}(B)\leq 1$.
Let $A$ be a regular local ring with $\operatorname{dim}(A)=2$ and
a localization of finite type over $B$.
\begin{description}
\item[(i)] (Proposition \ref{mixGS})
Suppose that $A$ is a mixed-characteristic and $(l, \operatorname{char}(A))=1$.
Then the sequence
\begin{equation*}
\operatorname{H}^{n+1}\left(
A_{et}, \mu_{l}^{\otimes n}
\right)\to
\operatorname{H}^{n+1}\left(
k(A)_{et}, \mu_{l}^{\otimes n}
\right)\to
\bigoplus_{\mathfrak{p}\in \spec A^{(1)}}
\operatorname{H}^{n}\left(
\kappa(\mathfrak{p})_{et}, \mu_{l}^{\otimes n-1}
\right)
\end{equation*}
is exact.
\item[(ii)] (Theorem \ref{DWE})
Suppose that $B$ 
is a discrete valuation ring which is an essentially of finite type over a field 
and $\operatorname{char}(B)=p>0$.
Suppose that $A$ is smooth over $B$.
Then the sequence
\begin{align*}
0\to
&\operatorname{H}^{1}(A_{et}, \nu_{r}^{n})\to
\operatorname{Ker}\Bigl(
\operatorname{H}^{1}(k(A)_{et}, \nu_{r}^{n})\to
\prod_{\mathfrak{p}\in \spec A^{(1)}}\operatorname{H}^{1}\left(
k(A_{\bar{\mathfrak{p}}})_{et}, \nu_{r}^{n}
\right)
\Bigr)\\
\to &
\bigoplus_{\mathfrak{p}\in \spec A^{(1)}}
\operatorname{H}^{1}(\kappa(\mathfrak{p})_{et}, \nu_{r}^{n-1})
\end{align*}
is exact. 
\end{description}
\end{thm}
Moreover we can consider a generalization of Artin's theorem (\cite[p.98, Theorem (3.1)]{Gr}) as follows.
\begin{prop}\upshape(Proposition \ref{GBsur})
Let $\mathcal{O}_{K}$ be the Henselization of a discrete valuation ring
which is an essentially of finite type over a field $k$. 
Let $\pi: \mathfrak{X}\to \spec \mathcal{O}_{K}$ 
be a proper and
smooth morphism and $Y$ the closed fiber of $\pi$. 
Suppose that $\operatorname{dim}(\mathfrak{X})=2$.

Then we can define a morphism
\begin{equation*}
\operatorname{H}^{n}_{\operatorname{B}}(\mathfrak{X})
\to
\operatorname{H}^{n}_{\operatorname{B}}(Y)
\end{equation*}
%
and this morphism is a surjective.
\end{prop}
\section{Notations}
For a scheme $X$, $X_{et}$ and $X_{Zar}$ denote the category of etale schemes over $X$ equipped with the etale and
Zariski topology, respectively. $X^{(i)}$ denotes the set of points of codimension $i$ and $X_{(i)}$ denotes the set of
points of dimension $i$. $\kappa(x)$ denotes the residue field of $x\in X$.

For $t\in \{\operatorname{et}, \operatorname{Zar}\}$, $\mathbb{S}_{X_{t}}$ denotes the category of sheaves on $X_{t}$.

For a scheme $X$ over $\mathbb{F}_{p}$, 
we denote the differential module 
$\Omega^{i}_{X/\mathbb{F}_{p}}$
simply by $\Omega^{i}_{X}$ and we denote
\begin{math}
\operatorname{Ker}
\left(
d: \Omega^{i}_{X}
\to
\Omega^{i+1}_{X}
\right)
\end{math}
simply by $Z\Omega^{i}_{X}$. 

Moreover we
denote the logarithmic Hodge-Witt sheaf
by $W_{n}\Omega^{i}_{X, \log}$ (cf.\cite[p.575, Definition 2.6.]{Sh})
and the logarithmic Hodge-Witt sheaf
$W_{1}\Omega^{i}_{X, \log}$
simply by
$\Omega^{i}_{X, \log}$.
\section{Definition of $\operatorname{H}^{i}_{\operatorname{B}}(X)$}\label{defgb}
Let $D_{i}=\mathbb{Z}[t_{0}, \cdots, t_{i}]/(\displaystyle\sum_{i}t_{i}-1)$, and $\Delta^{i}=\spec D_{i}$ be the 
algebraic $i$-simplex. For an equi-dimensional scheme $X$, let $z^{n}(X, i)$ be the free abelian group on closed
integral subschemes of codimension $n$ of $X\times \Delta^{i}$, which intersect all faces property. Intersecting with 
faces defines the structure of a simplicial abelian group, and hence gives a (homological) complex 
$z^{n}(X, *)$.
\vspace{0.5mm}\\

The complex of sheaves $\mathbb{Z}(n)_{t}$ on the site $X_{t}$, where $t\in \{et, Zar\}$, is defined as the cohomological 
complex with
$z^{n}(-, 2n-i)$ in degree $i$. 
For an abelian group $A$ we define $A(n)$ to be $\mathbb{Z}(n)\otimes A$.

Assume that $X$ is a smooth scheme of finite type over a field or a Dedekind ring, then there is a quasi-isomorphism
\begin{math}
\mathbb{Z}(1)\simeq \mathbb{G}_{m}[-1]. 
\end{math}

Moreover there is a quasi-isomorphism of complexes of etale sheaves

\begin{math}
\mathbb{Z}/l(n)_{et}\xrightarrow{\sim}\mu_{l}^{\otimes n}
\end{math}
for an integer $l$ which is invertible in $X$ by \cite[p.774, Theorem 1.2.4]{Ge} and \cite{V2}.
\vspace{1.0mm}\\

Let $\Gamma_{t}$ be the global sections functor from $\mathbb{S}_{X_{t}}$. Then the right derived functor 
$\mathbf{R}\Gamma_{t}$ exists on the derived category $\mathbf{D}(\mathbb{S}_{X_{t}})$. 
The (etale) motivic cohomology is defined as
\begin{equation*}
\operatorname{H}^{i}_{t}\left(X, A(n)\right)=\operatorname{H}^{i}\mathbf{R}\Gamma_{t}\left(A(n)\right).
\end{equation*}
If $t=Zar$ or $A(n)$ is isomorphic to a bounded below complex, then the hyper-derived functors 
$\mathbb{R}^{i}\Gamma_{t}\left(A(n)\right)$ 
are the cohomology of $\mathbf{R}\Gamma_{t}\left(A(n)\right)$:
\begin{equation*}
\mathbb{R}^{i}\Gamma_{t}\left(A(n)\right)=\operatorname{H}^{i}\mathbf{R}\Gamma_{t}\left(A(n)\right)
\end{equation*}
for all $i$.
\begin{defi}\upshape
Let $X$ be an equi-dimensional scheme. Then we define $\operatorname{H}^{i}_{\operatorname{B}}(X)$ as
\begin{equation*}
\operatorname{H}^{i}_{\operatorname{B}}(X)=\operatorname{H}^{i+1}_{et}(X, \mathbb{Z}(i-1))_{tor}.
\end{equation*}
\end{defi}
\begin{rem}\upshape
Let $X$ be a Noether regular integral scheme and $R(X)$ the ring of rational
functions on X. Then the canonical morphism
\begin{math}
\operatorname{Br}(X)\to
\operatorname{Br}(R(X)) 
\end{math}
is an injective by \cite[p.106, III, Example 2.22]{M} and $\operatorname{Br}(R(X))$ is a torsion group by 
\cite[p.53, II, Theorem 1.9]{M}.
Hence $\operatorname{Br}(X)$ is a torsion group.

But $\operatorname{H}^{n+2}_{et}(X, \mathbb{Z}(n))$ is not torsion group for any regular scheme. For example, by the relation
\begin{math}
\operatorname{H}^{i}_{et}(\spec \mathbf{P}^{m}_{K}, \mathbb{Z}(n))=\displaystyle\bigoplus^{m}_{j=0}\operatorname{H}^{i-2j}_{et}(\spec K, \mathbb{Z}(n-j))
\end{math}
for a positive integer $m$, we have
\begin{equation*}
\operatorname{H}^{4}_{et}(\spec \mathbf{P}^{m}_{K}, \mathbb{Z}(2))=\operatorname{H}^{0}_{et}(\spec K, \mathbb{Z}(0))=\mathbb{Z}
\end{equation*}
for an integer $m\geq 2$.
Hence 
\begin{math}
\operatorname{H}^{4}_{et}(\spec \mathbf{P}^{2}_{K}, \mathbb{Z}(2)) 
\end{math}
is not torsion group for an integer $m\geq 2$.
\end{rem}
\section{Results about Problem \ref{prob}}
In this section, we prove that Problem \ref{prob} is true on some cases.
\begin{lem}\upshape
\label{LDW}
(cf.\cite[Theorem 8.5]{Ge-L})
Let $B$ be an excellent regular ring which satisfies 
$\operatorname{dim}(B)\leq 1$ and $\operatorname{char}(B)=p>0$.

Let $X$ be an essentially smooth scheme over $\spec B$. Then
we have a quasi-isomorphism
\begin{equation}
\label{MDW}
\mathbb{Z}/p^{r}(n)_{Zar}\simeq \nu^{n}_{r}[-n]. 
\end{equation}
Here $\nu_{r}^{n}$ is the logarithmic de Rham-Witt sheaf 
$\operatorname{W\Omega}^{n}_{r, \log}$.
\end{lem}
\begin{proof}
We consider the following diagram, where the vertical maps are the symbol 
maps and the horizontal maps are the boundary maps:
\begin{equation*}
\begin{CD}
\displaystyle
\bigoplus_{x\in X^{(0)}}\operatorname{H}^{n}\left(
\kappa(x), \mathbb{Z}/p^{r}(n)
\right)
@>>>
\displaystyle
\bigoplus_{x\in X^{(1)}}\operatorname{H}^{n-1}\left(
\kappa(x), \mathbb{Z}/p^{r}(n-1)
\right)\\
@AAA @AAA\\
\displaystyle
\bigoplus_{x\in X^{(0)}}
\operatorname{K}^{M}_{n}(\kappa(x))/p^{r}
@>>>
\displaystyle
\bigoplus_{x\in X^{(1)}}
\operatorname{K}^{M}_{n-1}(\kappa(x))/p^{r}\\
@VVV @VVV\\
\displaystyle
\bigoplus_{x\in X^{(0)}}
\nu^{n}_{r}(\kappa(x))
@>>>
\displaystyle
\bigoplus_{x\in X^{(1)}}
\nu^{n-1}_{r}(\kappa(x))
\end{CD}
\end{equation*}

The upper square is commutative by \cite[Lemma 3.2]{Ge-L}
and the lower square is commutative by the definition of 
the boundary map (cf.\cite[p.605]{Sh}).
The vertical maps are isomorphisms by the theorem of
Nesterenko-Suslin (\cite[Theorem 4.9]{N-Su}) and 
the theorem of
Bloch-Gabber-Kato (\cite[p.117, Corollary (2.8)]{B-K}).

Moreover
\begin{align*}
\nu_{r}^{n}
=
\operatorname{Ker}\left(
\displaystyle
\bigoplus_{x\in X^{(0)}}
(i_{x})_{*}\nu_{r}^{n}\left(
\kappa(x)
\right)\to
\displaystyle
\bigoplus_{x\in X^{(1)}}
(i_{x})_{*}
\nu_{r}^{n-1}\left(
\kappa(x)
\right)
\right)
\end{align*}
by \cite[p.608, Theorem 5.2]{Sh} and 

\begin{align*}
&\mathcal{H}^{t}\left(
\mathbb{Z}/p^{r}(n)_{Zar}
\right)\\
=
&\operatorname{Ker}\left(
\displaystyle
\bigoplus_{x\in X^{(0)}}
(i_{x})_{*}\operatorname{H}^{t}\left(
\kappa(x), \mathbb{Z}/p^{r}(n)
\right)\to
\displaystyle
\bigoplus_{x\in X^{(1)}}
(i_{x})_{*}\operatorname{H}^{t-1}\left(
\kappa(x), \mathbb{Z}/p^{r}(n-1)
\right)
\right)
\end{align*}
by \cite[p.786, Corollary 4.5]{Ge}. Hence 
\begin{math}
\mathcal{H}^{t}\left(
\mathbb{Z}/p^{r}(n)_{Zar}
\right)
=0
\end{math}
for $t\neq n$ by \cite[Theorem 1.1]{Ge-L} and we have an isomorphism
\begin{math}
\mathcal{H}^{n}\left(
\mathbb{Z}/p^{r}(n)_{Zar}
\right)\simeq
\nu^{n}_{r}.
\end{math}
Therefore a composition of
maps of Zariski sheaves
\begin{align*}
&\mathbb{Z}/p^{r}(n)_{Zar}
\xleftarrow{}
\tau_{\leq n}\mathbb{Z}/p^{r}(n)_{Zar}
\xrightarrow{}
\mathcal{H}^{n}\left(
\tau_{\leq n}\mathbb{Z}/p^{r}(n)
\right)[-n]\\
\simeq
&\mathcal{H}^{n}\left(
\mathbb{Z}/p^{r}(n)
\right)[-n]
\simeq
\nu_{r}^{n}[-n]
\end{align*}
is a quasi-isomorphism and the statement follows.

\end{proof}
\begin{rem}\upshape
\label{direct}
Let $\{A_{\lambda}; f_{\mu\lambda}\}$ be a direct system of rings. 
Suppose that $\displaystyle\lim_{\substack{\to\\ \lambda}}A_{\lambda}=A$ is a Noether ring and
$f_{\mu\lambda}$ are flat. Then
\begin{equation*}
\lim_{\substack{\to\\ \lambda}} z^{n}(A_{\lambda}, *)=z^{n}(A, *). 
\end{equation*}
\end{rem}
\begin{prop}\upshape\label{Gysin}
Let 
$A$ be a discrete valuation ring which is 
an essentially of finite type over a field $k$ and 
$\operatorname{char}(A)=p>0$.

Let $(B, \mathfrak{n})$ be a local ring which is smooth scheme over $A$, 
$\operatorname{dim}(B)=d$ and 
$i: \spec(B/\mathfrak{n})\to \spec(B)$ the inclusion map.

Let the homomorphism
%
\begin{equation}\label{MGys}
\mathbb{Z}/p^{r}(n-d)_{et}[-2d]\to
\mathbf{R}i^{!}\mathbb{Z}/p^{r}(n)_{et}
\end{equation}
be induced by adjointness from the natural inclusion map. 

Then the homomorphism
\begin{equation*}
\operatorname{H}^{n-d}
\left(
\spec(B)_{et}, \mathbb{Z}/p^{r}(n-d)
\right)
\to
\operatorname{H}^{n+d}_{\spec(B/\mathfrak{n})}
\left(
\spec(B)_{et}, \mathbb{Z}/p^{r}(n)
\right)
\end{equation*}
which is defined by the map (\ref{MGys}) agrees with the canonical map
\begin{equation*}
\operatorname{H}^{0}\left(
\spec(B/\mathfrak{n})_{et},
\operatorname{W\Omega}^{n-d}_{r, \log}
\right)
\to
\operatorname{H}^{d}_{\spec(B/\mathfrak{n})}\left(
\spec(B)_{et}, 
\operatorname{W\Omega}^{n}_{r, \log}
\right)
\end{equation*}
which is defined in \cite{Sh} up to sign
via the 
quasi-isomorphism $(1)$
in Lemma \ref{LDW}.
\end{prop}
\begin{proof}\upshape
By Quillen's method (cf.\cite[Proof of Theorem 5.11]{Q}), 
it is sufficient to prove that the statement is true in
the case where 
$B$ is an essentially smooth over a perfect field.
Therefore the statement follows from 
\cite[Proposition 2.3.1]{Sat}, \cite[Theorem 2.15]{Z}
and \cite[Theorem 2.16]{Z}. 
\end{proof}
\begin{prop}
\upshape
\label{leD}
Let $X$ be a local ring at a point $x$ of an essentially smooth scheme over a Dedekind ring. Then Problem \ref{prob} (i) is true.
\end{prop}
\begin{proof}
Since 
\begin{math}
\mathcal{H}^{n+1}\left(
\mathbb{Z}(n)_{Zar}
\right)
\end{math}
is the Zariski sheaf associated to the presheaf
\begin{equation*}
U\mapsto \operatorname{H}^{n+1}\left(
\mathbb{Z}(n)_{Zar}(U)
\right),
\end{equation*}
the stalk of 
\begin{math}
\mathcal{H}^{n+1}\left(
\mathbb{Z}(n)_{Zar}
\right)
\end{math}
at $y$, 
\begin{math}
\mathcal{H}^{n+1}(\mathbb{Z}(n)_{Zar})_{y}, 
\end{math}
is equal to
\begin{equation*}
\lim_{\substack{\to\\ U}}\operatorname{H}^{n+1}\left(
\mathbb{Z}(n)_{Zar}(U)
\right)
\end{equation*}
for $y\in X$
where $U$ runs over all open subschemes of $X$
which contain $y$. 

Since $X$ is the spectrum of a local ring and
$x$ is the closed point of $X$, we have

\begin{equation*}
\mathcal{H}^{n+1}\left(
\mathbb{Z}(n)_{Zar}
\right)_{x}
=
\operatorname{H}^{n+1}\left(
\mathbb{Z}(n)_{Zar}(X)
\right)
=
\operatorname{H}^{n+1}_{Zar}\left(
X, \mathbb{Z}(n)
\right)
\end{equation*}
by \cite[p.779, Theorem 3.2 (b)]{Ge}. Moreover, 
\begin{equation*}
\operatorname{H}^{n+1}_{et}\left(X, \mathbb{Z}(n)\right) 
=
\operatorname{H}^{n+1}_{Zar}\left(X, \mathbb{Z}(n)\right)
\end{equation*}
by \cite[p.774, Theorem 1.2.2]{Ge} and \cite{V2}. Therefore 
\begin{equation*}
\operatorname{H}^{n+1}_{et}\left(X, \mathbb{Z}(n)\right)=0 
\end{equation*}
by \cite[p.786, Corollary 4.4]{Ge}. 

Let $m$ be a positive integer. Then we have
a distinguished triangle 
\begin{equation}
\label{disl}
\cdots\to\mathbb{Z}(n)_{et}\xrightarrow{\times m}\mathbb{Z}(n)_{et}\to \mathbb{Z}/m(n)_{et}
\to\cdots.
\end{equation}
Hence the statement follows from \cite[p.774, Theorem 1.2.4]{Ge} and Lemma \ref{LDW}.
\end{proof}
\begin{rem}\upshape
\label{pic}
Let $X$ be the spectrum of a regular local ring. Then

\begin{equation*}
\operatorname{H}^{1}\left(
X_{Zar}, \mathcal{O}_{X}^{*}
\right)=0 
\end{equation*}
by \cite[p.131, II, Proposition 6.2]{Ha} and \cite[p.141, II, Proposition 6.11]{Ha}, 
because a regular local ring is UFD. 

Moreover the canonical map
\begin{math}
\operatorname{H}^{1}\left(
X_{Zar}, \mathcal{O}_{X}^{*}
\right)
\to
\operatorname{H}^{1}\left(
X_{et}, \mathbb{G}_{m}
\right)
\end{math}
is an isomorphism by \cite[p.124, Proposition 4.9]{M}. Hence
\begin{math}
\operatorname{H}^{1}\left(
X_{et}, \mathbb{G}_{m}
\right)=0.
\end{math}

Therefore Property \ref{per} (i-a) follows from the exact sequence
\begin{equation*}
0\to \mu_{l}\to \mathbb{G}_{m}
\xrightarrow{\times l}
\mathbb{G}_{m}
\to 0
\end{equation*}
for an integer $l$ which is prime to $\operatorname{char}(X)$
(\cite[p.66, II. Example 2.18 (b)]{M}).

Let $p$ be a prime integer. Then we have an exact sequence 
\begin{equation}
\label{bl}
0\to 
\mathbb{G}_{m}
\xrightarrow{\times p^{n}}
\mathbb{G}_{m}
\xrightarrow{\operatorname{dlog}}
\operatorname{W\Omega}^{1}_{n, \operatorname{log}}
\to 0,
\end{equation}
in the case where $X$ is a smooth scheme over $\mathbb{F}_{p}$
by \cite[p.580, Proposition 3.23.2]{I}. Since
the sheaves
\begin{math}
\mathbb{G}_{m}
\end{math}
and
\begin{math}
\operatorname{W\Omega}^{1}_{n, \operatorname{log}}
\end{math}
are commute with direct limits,
we have the exact sequence (\ref{bl}) in the case where 
$X$ is the spectrum of a regular local ring with $\operatorname{char}(X)=p>0$
by the theorem of Popescu (\cite[p.90, (2.7) Corollary]{Po}).

Therefore Property \ref{per} (i-b) follows.

\end{rem}
\begin{lem}\upshape
\label{tru}
Let $F: \mathcal{A}\to \mathcal{B}$ be a left exact functor between abelian categories and
let $\mathcal{A}$ be a Grothendieck category. Then
\begin{equation*}
\mathbf{R}^{n+1}F(\tau_{\leq n}A^{\bullet})
=\operatorname{Ker}\Bigl(\mathbf{R}^{n+1}F(A^{\bullet})\to F(\mathcal{H}^{n+1}(A^{\bullet}))
\Bigr)
\end{equation*}
for any complex $A^{\bullet}$.
\end{lem}
\begin{proof}
If $B^{\bullet}$ is bounded below, there is a convergent spectral sequence for 
the hyper-cohomology with 
\begin{equation}
\label{sshc}
\mathbb{R}^{p}F(\mathcal{H}^{q}(B^{\bullet}))\Rightarrow \mathbb{R}^{p+q}F(B^{\bullet})
\end{equation}
and 
\begin{equation*}
\mathbf{R}^{n}F(B^{\bullet})=\mathbb{R}^{n}F(B^{\bullet}). 
\end{equation*}
So 
\begin{align*}
\mathbf{R}^{n}F(\tau_{\geq n+1}A^{\bullet})&=\mathbb{R}^{n}F(\tau_{\geq n+1}A^{\bullet})=0,\\ 
\mathbf{R}^{n+1}F(\tau_{\geq n+1}A^{\bullet})&=\mathbb{R}^{n+1}F(\tau_{\geq n+1}A^{\bullet})\xrightarrow{\sim}
F(\mathcal{H}^{n+1}(A^{\bullet}))
\end{align*}
for any complex $A^{\bullet}$. By a distinguished triangle
\begin{equation*}
\cdots\to\tau_{\leq n}A^{\bullet}\to A^{\bullet}\to \tau_{\geq n+1}A^{\bullet}\to\cdots,
\end{equation*}
we have a distinguished triangle
\begin{equation*}
\cdots\to\mathbf{R}F\left(\tau_{\leq n}A^{\bullet}\right)\to 
\mathbf{R}F\left(A^{\bullet}\right)\to 
\mathbf{R}F\left(\tau_{\geq n+1}A^{\bullet}\right)\to \cdots.
\end{equation*}
Therefore,
\begin{equation*}
\mathbf{R}^{n+1}F(\tau_{\leq n}A^{\bullet})
=\operatorname{Ker}\Bigl(\mathbf{R}^{n+1}F(A^{\bullet})\to F(\mathcal{H}^{n+1}(A^{\bullet}))
\Bigr).
\end{equation*}
\end{proof}
\begin{rem}\upshape
Let $A^{\bullet}$ be a bounded below complex.
Since the edge maps of spectral sequence are natural maps, the morphism
\begin{equation*}
\mathbf{R}^{n+1}F(A^{\bullet})\to \mathbf{R}^{n+1}F(\tau_{\geq n+1}A^{\bullet})
\end{equation*}
corresponds to the edge map of 
\begin{math}
\mathbb{R}^{p}F(\mathcal{H}^{q}(A^{\bullet}))\Rightarrow \mathbb{R}^{p+q}F(A^{\bullet}).
\end{math}
\end{rem}

\begin{prop}\upshape
\label{tru2}
Let $B$ be a regular local ring with $\operatorname{dim}(B)\leq 1$, $\mathfrak{X}$ an essentially of finite type scheme over 
$\spec B$ and $i:Y\to \mathfrak{X}$ a closed
subscheme of codimension $c$ with open complement $j:X\to \mathfrak{X}$.

Suppose that $X$ is an essentially smooth over a regular ring of dimension at most one. 
Then we have a quasi-isomorphism
\begin{equation}\label{puri}
\tau_{\leq n+2}\Bigl(\mathbb{Z}(n-c)_{et}[-2c]\Bigr)
\xrightarrow{\sim} 
\tau_{\leq n+2}\mathbf{R}i^{!}\mathbb{Z}(n)_{et}
\end{equation}
and a quasi-isomorphism
\begin{equation}
\label{quasiil}
\tau_{\leq n+1}\Bigl(\mathbb{Z}/m(n-c)_{et}[-2c]\Bigr)
\xrightarrow{\sim}\tau_{\leq n+1}\mathbf{R}i^{!}\mathbb{Z}/m(n)_{et}
\end{equation}
for any positive integer $m$.

Moreover, if 
\begin{math}
\mathbf{R}^{i}j_{*}\mathbb{Z}(n)_{Zar}=0
\end{math}
for $i\geq n+2$
,we have a distinguished triangle
\begin{equation}\label{diset}
\cdots \to i_{*}\mathbb{Z}(n-c)_{et}[-2c] \to \mathbb{Z}(n)_{et} \to 
\tau_{\leq n+1}\mathbf{R}j_{*}\mathbb{Z}(n)_{et}\to \cdots
\end{equation}

\end{prop}
\begin{proof}(cf. The proof of \cite[Theorem 1.2.1]{Ge})
We have a quasi-isomorphism
\begin{equation}
\label{Z-ej}
\tau_{\leq n+1}\epsilon^{*}\mathbf{R}j_{*}\mathbb{Z}(n)_{Zar}
\xrightarrow{\sim}
\tau_{\leq n+1}\epsilon^{*}\mathbf{R}j_{*}\mathbf{R}\epsilon_{*}\mathbb{Z}(n)_{et}
\xrightarrow{\sim}
\tau_{\leq n+1}\mathbf{R}j_{*}\mathbb{Z}(n)_{et}
\end{equation}
by \cite[p.774, Theorem 1.2.2]{Ge} (cf. \cite[p.787]{Ge}). Since
\begin{equation*}
\mathbf{R}^{n+2}j_{*}\left(\tau_{\leq n+1}\mathbf{R}\epsilon_{*}\mathbb{Z}(n)_{et}\right)
\to \mathbf{R}^{n+2}j_{*}\left(\mathbf{R}\epsilon_{*}\mathbb{Z}(n)_{et}\right)
\end{equation*}
is an injective, the composite map 
\begin{align*}
&\epsilon^{*}\mathbf{R}^{n+2}j_{*}\mathbb{Z}(n)_{Zar}
\xrightarrow{\sim}
\epsilon^{*}\mathbf{R}^{n+2}j_{*}\left(\tau_{\leq n+1}\mathbf{R}\epsilon_{*}\mathbb{Z}(n)_{et}\right) \\
\hookrightarrow
&\epsilon^{*}\mathbf{R}^{n+2}j_{*}\mathbf{R}\epsilon_{*}\mathbb{Z}(n)_{et}
\xrightarrow{\sim}
\mathbf{R}^{n+2}j_{*}\mathbb{Z}(n)_{et}
\end{align*}
is an injective.
Moreover we have the map of 
distinguished triangles
\begin{equation*}
\begin{CD}
\epsilon^{*}\mathbb{Z}(n-c)_{Zar}[-2c]@>>>\epsilon^{*}i^{*}\mathbb{Z}(n)_{Zar}
@>>> \epsilon^{*}i^{*}\mathbf{R}j_{*}\mathbb{Z}(n)_{Zar} \\
@VVV @| @VVV \\
\mathbf{R}i^{!}\mathbb{Z}(n)_{et}@>>> i^{*}\mathbb{Z}(n)_{et} @>>> 
i^{*}\mathbf{R}j_{*}\mathbb{Z}(n)_{et}.
\end{CD}
\end{equation*}
Therefore we have the quasi-isomorphism (\ref{puri}) by the five-lemma.
Applying $\epsilon^{*}$ to the distinguished triangle of \cite[p.780, Corollary 3.3 (a)]{Ge},
we get a distinguished triangle
\begin{equation*}
i_{*}\mathbb{Z}(n-c)_{et}[-2c]\to \mathbb{Z}(n)_{et}\to \epsilon_{*}\mathbf{R}j_{*}\mathbb{Z}(n)_{Zar}.
\end{equation*}
Hence we have the distinguished triangle (\ref{diset}) by the quasi-isomorphism (\ref{Z-ej}).
The quasi-isomorphism (\ref{quasiil}) follows from the quasi-isomorphism (\ref{puri}).

\end{proof}
\begin{rem}\upshape
If $B$ is a Dedekind ring  and $i$ is the inclusion of one of the
closed fibers, we have the quasi-isomorphisms (\ref{puri}) and (\ref{quasiil})
as \cite[p.774, Theorem 1.2.1]{Ge}.
\end{rem}
\begin{prop}\upshape
\label{GSV}
Let $\mathfrak X$, $Y$ and $X$ be same as above. Suppose that $c=1$.
Then we have a quasi-isomorphism
\begin{equation}
\tau_{\leq n+1}\left(\mathbb{Z}(n)^{\mathfrak{X}}_{Zar}\right)
\xrightarrow{\sim}
\tau_{\leq n+1}\left(\mathbf{R}\epsilon_{*}\mathbb{Z}(n)^{\mathfrak{X}}_{et}\right)
\end{equation}
in the following cases:
\begin{description}
\item[(i)] $\mathfrak{X}$ is a reduced scheme and finite type over a field with 
$\operatorname{dim}(X)=1$,
\item[(ii)] $\mathfrak{X}$ is a regular scheme with $\operatorname{dim}(\mathfrak{X})=2$.
\end{description}
\end{prop}
\begin{proof}(cf. The proof of \cite[Theorem 1.2.2]{Ge})
We can prove the statement as \cite[p.774, Theorem 1.2.2]{Ge}. It is sufficient to show that
\begin{equation*}
\tau_{\leq n+2}\left(i_{*}\mathbb{Z}(n-1)[-2]_{Zar}\right)
\to
\tau_{\leq n+2}\left(\mathbf{R}\epsilon_{*}i_{*}\mathbf{R}i^{!}\mathbb{Z}(n)_{et}\right)
\end{equation*}
is a quasi-isomorphism.

We have a quasi-isomorphism
\begin{align*}
&\Bigl(\tau_{\leq n}i_{*}\mathbf{R}\epsilon_{*}\mathbb{Z}(n-1)_{et}\Bigr)[-2] 
\simeq \tau_{\leq n+2}\Bigl(i_{*}\mathbf{R}\epsilon_{*}\mathbb{Z}(n-1)_{et}[-2]\Bigr)\\
\simeq &\tau_{\leq n+2}\Bigl(i_{*}\mathbf{R}\epsilon_{*}\left(\tau_{\leq n+2}\Bigl(\mathbb{Z}(n-1)_{et}[-2]\Bigr)\right)\Bigr)\\
\xrightarrow{\sim}
&\tau_{\leq n+2}\Bigl(i_{*}\mathbf{R}\epsilon_{*}\Bigl(\tau_{\leq n+2}\Bigl(\mathbf{R}i^{!}\mathbb{Z}(n)_{et}\Bigr)\Bigr)\Bigr)
\simeq \tau_{\leq n+2}\Bigl(i_{*}\mathbf{R}\epsilon_{*}\mathbf{R}i^{!}\mathbb{Z}(n)_{et}\Bigr)
\end{align*}
by Proposition \ref{tru2}. 

Suppose that $\mathfrak{X}$ satisfies the condition (i). Then we can choose $X, Y$ as 
satisfying the condition of Proposition \ref{tru2}
by a resolution of singularities of curves.
Hence we have a quasi-isomorphism
\begin{align}
&\tau_{\leq n+2}\Bigl(i_{*}\mathbb{Z}(n-1)_{Zar}[-2]\Bigr)\simeq 
\Bigl(\tau_{\leq n}i_{*}\mathbb{Z}(n-1)_{Zar}\Bigr)[-2] \\
\xrightarrow{\sim}
&\Bigl(\tau_{\leq n}i_{*}\mathbf{R}\epsilon_{*}\mathbb{Z}(n-1)_{et}\Bigr)[-2]
\end{align}
by \cite[Theorem 1.2.2]{Ge}. Therefore the statement follows if $\mathfrak{X}$ satisfies
the condition (i).
\end{proof}
Before we prove that Problem \ref{prob} (iii) is true for some cases, we prove the following lemma.
\begin{lem}\upshape
\label{str}
Let B be a regular ring with $\operatorname{dim}(B)\leq 1$
and X an essentially of finite type over $\spec B$. 
Let $i:Y\to X$ be a closed subscheme of codimension $1$ with open complement
$j: U\to X$. Suppose that $U$ is a regular scheme with
\begin{math}
\operatorname{dim}(U)\leq 1. 
\end{math}

Then we have
\begin{equation*}
\mathbf{R}^{n+2}j_{*}\left(
\mathbb{Z}(n)_{et}
\right)_{\bar{x}}=
\operatorname{H}^{n+1}_{B}\left(
U\times_{X}\spec\mathcal{O}_{X, \bar{x}}
\right)
\end{equation*}
for $x\in X$.
\end{lem}
\begin{proof}

Since 
\begin{math}
\mathbb{Z}/m(n)_{et}^{U} 
\end{math}
is bounded below for any positive integer $m$, 
we have a spectral sequence
\begin{equation*}
\mathbf{R}^{p}j_{*}\left(
\mathcal{H}^{q}
(\mathbb{Z}/m(n)_{et})
\right)
\Rightarrow
\mathbf{R}^{p+q}j_{*}\mathbb{Z}/m(n).
\end{equation*}
Moreover Bloch's cycle complex $\mathbb{Z}(n)$
commutes with direct limits by Remark \ref{direct},
so $\mathbb{Z}/m(n)_{et}$ does.

Hence we have
\begin{equation*}
\left(
\mathbf{R}^{n+1}j_{*}\mathbb{Z}/m(n)_{et}
\right)_{\bar{x}}
=
\operatorname{H}^{n+1}_{et}\left(
U\times_{X}\mathcal{O}_{X, \bar{x}}, \mathbb{Z}/m(n)
\right)
\end{equation*}
by \cite[p.88, III, Lemma 1.16]{M}. Moreover we have
\begin{align*}
\mathbf{R}^{n+1}j_{*}\left(
\mathbb{Z}(n)_{et}
\right)_{\bar{x}}&=
\left(\epsilon^{*}\mathbf{R}^{n+1}j_{*}
\mathbb{Z}(n)_{Zar}
\right)_{\bar{x}}=
\operatorname{H}^{n+1}_{Zar}\left(
U\times_{X}\spec\mathcal{O}_{X, \bar{x}}, \mathbb{Z}(n)
\right)\\
&=
\operatorname{H}^{n+1}_{et}\left(
U\times_{X}\spec\mathcal{O}_{X, \bar{x}}, \mathbb{Z}(n)
\right)
\end{align*}
by the quasi-isomorphism (\ref{Z-ej}) and \cite[p.779, Theorem 3.2 b)]{Ge}. 
Therefore the statement follows.
\end{proof}
\begin{prop}\upshape
\label{supp}
Let $X$ be a regular integral scheme with $\operatorname{dim}(X)$=1. Then Problem \ref{prob}
(iii) is true for $X$.
\end{prop}
\begin{proof}
We have
\begin{equation*}
\operatorname{H}^{n}_{et}\left(
\spec k(X), \mathbb{Z}(n-1)
\right)
=\lim_{\substack{\to\\U}}\operatorname{H}^{n}_{et}\left(
U, \mathbb{Z}(n-1)
\right)
\end{equation*}
by \cite[p.774, Theorem 1.2.2]{Ge} and \cite[p.779, Theorem 3.2 b)]{Ge}.
Here the limit is taken over any open subset of $X$.

Moreover we have
\begin{equation*}
\operatorname{H}^{n}_{et}\left(
\spec k(X), \mathbb{Z}/m(n-1)
\right)
=\lim_{\substack{\to\\U}}
\operatorname{H}^{n}_{et}\left(U, \mathbb{Z}/m(n-1)
\right)
\end{equation*}
for any positive integer $m$ as the proof of Lemma \ref{str}.
Hence we have
\begin{align}\label{limU}
\operatorname{H}^{n+1}_{et}\left(\spec k(X), \mathbb{Z}(n-1)\right)_{m}
&=\lim_{\substack{\to\\U}}\operatorname{H}^{n+1}_{et}\left(U, \mathbb{Z}(n-1)
\right)_{m}
\end{align}
by the distinguished triangle (\ref{disl}). 

Let $k$ be a field. Then we have
\begin{equation*}
\operatorname{H}^{n}_{et}\left(
\spec k, \mathbb{Z}(n-1)
\right)=
\operatorname{H}^{n}_{Zar}\left(
\spec k, \mathbb{Z}(n-1)
\right)
=
\operatorname{H}_{n-2}\left(
z^{n-1}(\spec k, *)
\right)
=0
\end{equation*}
by \cite[p.774, Theorem 1.2.2]{Ge} and \cite[p.779, Theorem 3.2 b)]{Ge}. Hence the natural 
homomorphism

\begin{equation*}
\operatorname{H}^{n}_{B}\left(X\right)\to
\operatorname{H}^{n}_{B}(U)
\end{equation*}
is an injective by Proposition \ref{tru2}. Therefore
\begin{equation*}
\operatorname{H}^{n}_{B}(X)\to
\operatorname{H}^{n}_{B}\left(k(X)\right)
\end{equation*}
is an injective by (\ref{limU}).

For each $x\in X$, we choose a geometric point 
$u_{x}: \bar{x}\to X$. Then 
\begin{math}
F\to \displaystyle\prod_{x\in X} (u_{x})_{*}(u_{x})^{*}F 
\end{math}
is an injective for a sheaf on $X_{et}$ (\cite[p.90, III, Remark 1.20 (c)]{M}). So
\begin{align*}
&\operatorname{H}^{n+2}_{et}\left(
X, \tau_{\leq n+1}\mathbf{R}j_{*}\mathbb{Z}(n)
\right) \\
=&\operatorname{Ker}\Bigl(
\operatorname{H}^{n+2}_{et}\left(
U, \mathbb{Z}(n)
\right)\to 
\prod_{x\in X}\left(
\mathbf{R}^{n+2}j_{*}\mathbb{Z}(n)
\right)_{\bar{x}}
\Bigr)\\
=&\operatorname{Ker}\Bigl(
\operatorname{H}^{n+2}_{et}\left(
U, \mathbb{Z}(n)
\right)\to
\prod_{x\in X}
\operatorname{H}^{n+2}_{et}\left(
k(\mathcal{O}_{X, \bar{x}}), \mathbb{Z}(n)
\right)
\Bigr)
\end{align*}
for any open immersion $j: U\to X$. Therefore the statement follows.

\end{proof}
\begin{rem}\upshape
Let 
\begin{math}
E^{l, m}_{2}\Rightarrow E^{l+m}
\end{math}
be a spectaral sequence. Suppose that $E^{l, m}_{2}=0$ for $l<0$ or $m<0$.
Then we can define the morphism
\begin{equation}\label{spe}
\operatorname{Ker}\left(E^{n}\to E^{0, n}\right)\simeq E^{n}_{1}
\twoheadrightarrow E^{n}_{1}/E^{n}_{2}\simeq
E^{1, n-1}_{\infty}\hookrightarrow E^{1, n-1}_{2}.
\end{equation}
Moreover, the morphism (\ref{spe}) is an isomorphism if $E^{l, m}_{2}=0$ for $l\geq 2$.

Let $(A, \mathfrak{m}, k)$ be a discrete valuation ring. 
Let $j: \spec K\to \spec A$ be the generic point and
$i: \spec k\to \spec A$ the closed immersion. 

For the above,
\begin{align*}
&\operatorname{Ker}\Bigl(\operatorname{H}^{n+1}_{et}(\spec K, \mathbb{Z}/p(n))
\to \operatorname{H}^{n+1}_{et}(\spec K_{\bar{\mathfrak m}}, \mathbb{Z}/p(n))\Bigr)\\
=&\operatorname{H}^{1}_{et}\left(
\spec A, \mathbf{R}^{n}j_{*}\mathbb{Z}/p(n)
\right)
=\operatorname{H}^{1}_{et}\left(
\spec A, i_{*}i^{*}\mathbf{R}^{n}j_{*}\mathbb{Z}/p(n)
\right) \\
=&\operatorname{H}^{1}\left(
\operatorname{Gal}(k_{s}/k), 
\operatorname{H}^{n}_{et}(\spec K_{\bar{\mathfrak m}}, \mathbb{Z}/p(n))
\right)
\end{align*}
for $\operatorname{char}(k)=p>0$ and $n>0$ where $k_{s}$ is the separable closure
of $k$ and $K_{\bar{\mathfrak{m}}}$ is the maximal unramified extension of $K$.

We have the commutative diagram
\begin{equation}
\begin{CD}
@. 
\operatorname{H}^{n+1}_{et}
\left(
K, \mathbb{Z}/p(n)
\right)
\\
@. @|
\\
\operatorname{H}^{1}_{et}
\left(
X, \mathbf{R}^{n}j_{*}\mathbb{Z}/p(n)
\right)
@>>>
\operatorname{H}^{n+1}_{et}
\left(
X, \mathbf{R}j_{*}\mathbb{Z}/p(n)
\right)
\\
@VVV @VVV \\
\operatorname{H}^{1}_{et}
\left(
X, i_{*}\mathbf{R}^{n+1}i^{!}\mathbb{Z}/p(n)
\right)
@>>>
\operatorname{H}^{n+1}
\left(
X_{et}, i_{*}\mathbf{R}i^{!}\mathbb{Z}/p(n)[+1]
\right)\\
@. @| \\
@.
\operatorname{H}^{n+2}_{\mathfrak{m}}\left(
K, \mathbb{Z}/p(n)
\right)
\end{CD}
\end{equation}
where the horizontal maps are given by spectral sequences

\begin{equation*}
\operatorname{H}^{l}_{et}\left(
X, \mathbf{R}^{m}j_{*}\mathbb{Z}/p(n)
\right)
\Rightarrow
\operatorname{H}^{l+m}_{et}\left(
K, \mathbb{Z}/p(n)
\right)
\end{equation*}
and

\begin{equation*}
\operatorname{H}^{l}_{et}\left(
X, i_{*}\mathbf{R}^{m}i^{!}\mathbb{Z}/p(n-1)
\right)
\Rightarrow
\operatorname{H}^{l+m}_{\mathfrak{m}}\left(
K, \mathbb{Z}/p(n)
\right).
\end{equation*}
%


Moreover, we have the commutative diagram
\begin{equation*}
\begin{CD}
\epsilon^{*}i^{*}\mathbf{R}j_{*}\mathbb{Z}/p(n)_{Zar}
@>>>
\epsilon^{*}\mathbb{Z}/p(n-1)_{Zar}[-1] \\
@VVV @VVV  \\
i^{*}\tau_{\leq n}\left(
\mathbf{R}j_{*}\mathbb{Z}/p(n)_{et}
\right)
@>>>
\tau_{\leq n+1}\left(
\mathbf{R}i^{!}\mathbb{Z}/p(n)_{et}
\right)
[+1]
\end{CD}
\end{equation*}
where the vertical maps are quasi-isomorphism and the 
homomorphism
\begin{equation*}
\left(
\epsilon^{*}i^{*}\mathcal{H}^{n}\left(
\mathbf{R}j_{*}\mathbb{Z}/p(n)_{Zar}
\right)
\right)_{\bar{x}}
\to
\mathcal{H}^{n-1}\left(
\epsilon^{*}\mathbb{Z}/p(n-1)_{Zar}
\right)_{\bar{x}}
\end{equation*}
agrees with the symbol map
\begin{math}
K^{n}_{M}(
K_{\bar{\mathfrak{m}}}
)/p
\to
K^{n-1}_{M}(
k_{s})
/p
\end{math}
by \cite[Lemma 3.2]{Ge-L}.

Therefore the homomorphism
\begin{equation*}
\operatorname{H}^{n+1}_{et}\left(
X, \tau_{\leq n}\left(
\mathbf{R}j_{*}\mathbb{Z}/p(n)
\right)
\right)
\to
\operatorname{H}^{n}_{et}\left(
X, i_{*}\left(
\mathbb{Z}/p(n-1)
\right)
\right)
\end{equation*}
which is induced by the map
\begin{equation*}
\tau_{\leq n}\left(
\mathbf{R}j_{*}\mathbb{Z}/p(n)_{et}
\right)
\to
i_{*}\left(
\mathbb{Z}/p(n-1)_{et}
\right)[-1]
\end{equation*}
agrees with the homomorphism
\begin{equation}\label{Ka}
\operatorname{H}^{1}\left(
\operatorname{Gal}(k_{s}/k),
K^{n}_{M}(K_{\bar{\mathfrak{m}}})
/p
\right)
\to
\operatorname{H}^{1}\left(
\operatorname{Gal}(k_{s}/k),
K^{n-1}_{M}(k_{s})/p
\right)
\end{equation}
which is induced by the symbol map. 

In the case where $[k: k^{p}]\leq n-1$, the homomorphism (\ref{Ka}) is defined by K.Kato (\cite[p.150, \S1, (1.3) (ii)]{K}).
\end{rem}
\begin{prop}\upshape
\label{rl2}
Let $B$ be a regular local ring with $\operatorname{dim}(B)\leq 1$.
Let $\spec A$ be a regular local scheme with $\operatorname{dim}(A)=2$
and a localization of finite type over $\spec B$.

Then Problem \ref{prob} (iii) is true for $\spec A$.
\end{prop}
\begin{proof}
Let $\spec A/(f)\to \spec A$ be a regular closed subscheme of codimension $1$.
Then the morphism
\begin{equation*}
\operatorname{H}^{n+2}_{et}\left(A, \mathbb{Z}(n)\right)\to 
\operatorname{H}^{n+2}_{et}\left(A_{f}, \mathbb{Z}(n)\right)
\end{equation*}
is an injective. Hence
\begin{equation*}
\operatorname{H}^{n+2}_{et}\left(A, \mathbb{Z}(n)\right)\to
\operatorname{H}^{n+2}_{et}\left(k(A), \mathbb{Z}(n)\right)
\end{equation*}
is an injective.

Let 
\begin{math}
\spec A/\mathfrak{m}\to 
\spec A 
\end{math}
be the closed point with the open complement 
$j: U\to X$.

Then
the morphism
\begin{equation*}
\operatorname{H}^{n+2}_{et}\left(A, \mathbb{Z}(n)
\right)\to
\operatorname{H}^{n+2}_{et}\left(A, \tau_{\leq n+1}
\mathbf{R}j_{*}\mathbb{Z}(n)
\right)
\end{equation*}
is a surjective. Moreover we have
\begin{equation*}
\operatorname{H}^{n+2}_{et}\left(
A, \tau_{\leq n+1}
\mathbf{R}j_{*}\mathbb{Z}(n)
\right)=
\operatorname{Ker}\bigl(
\operatorname{H}^{n+2}_{et}(U, \mathbb{Z}(n))
\to
\Gamma(X, \mathbf{R}^{n+2}j_{*}\mathbb{Z}(n)_{et})
\bigr)
\end{equation*}
Therefore the statement follows.
\end{proof}
\begin{lem}\upshape
\label{mixl}
Let $A$ be same as above.
Suppose that $B$ be a discrete valuation ring with mixed-characteristic. Let $l$ be a positive integer 
such that 
\begin{math}
(l, \operatorname{char}(A))=1. 
\end{math}

Then we have
\begin{equation*}
\mathbb{Z}/l(n)_{et}
\simeq 
\mu_{l}^{\otimes n}. 
\end{equation*}
\end{lem}
\begin{proof}
We have a quasi-isomorphism
\begin{equation*}
\tau_{\leq n+1}\Bigl(
\mathbb{Z}/l(n)
\Bigr)
\simeq
\tau_{\leq n+1}\Bigl(
\mu^{\otimes n}_{l}
\Bigr)
\end{equation*}
by Proposition \ref{tru2} and the absolute purity theorem (\cite{Ga}). Moreover 
\begin{equation*}
\mathcal{H}^{i}\left(
\mathbb{Z}(n)_{Zar} 
\right)
=0
\end{equation*}
for $i\geq n+2$,
so we have 
\begin{math}
\mathcal{H}^{n+2}\left(
\mathbb{Z}/l(n)_{et}
\right)=0.
\end{math}
Therefore the statement follows.

\end{proof}
\begin{prop}\upshape
\label{mixGS}
Let notations be same as Lemma \ref{mixl}. Then the sequence
\begin{equation}
\label{gp}
\operatorname{H}^{n+1}\left(
A_{et}, \mu_{l}^{\otimes n}
\right)\to
\operatorname{H}^{n+1}\left(
k(A)_{et}, \mu_{l}^{\otimes n}
\right)\to
\bigoplus_{\mathfrak{p}\in \spec A^{(1)}}
\operatorname{H}^{n}\left(
\kappa(\mathfrak{p})_{et}, \mu_{l}^{\otimes n-1}
\right)
\end{equation}
is exact.
\end{prop}
\begin{proof}
Let 
\begin{math}
i: \spec A/(f) \to \spec A 
\end{math}
be a regular closed subscheme of codimension $1$.
Let 
\begin{math}
j: \spec A_{f}\to \spec A 
\end{math}
be the open complement of $i$.
Then 
\begin{math}
\mathbf{R}^{n+1}j_{*}\mu_{l}^{\otimes n}=0 
\end{math}
for $n+1 \neq 1$ by the absolute purity theorem (\cite{Ga}). Hence
\begin{math}
\operatorname{H}^{n+1}(A_{et}, \tau_{\leq n}\mathbf{R}j_{*}\mu_{l}^{n})=
\operatorname{H}^{n+1}((A_{f})_{et}, \mu_{l}^{\otimes n}),
\end{math}
so we have 
\begin{equation*}
\operatorname{H}^{n+2}_{et}\left(
A, \tau_{\leq n+1}\mathbf{R}j_{*}\mathbb{Z}(n)
\right)_{l}=
\operatorname{H}^{n+2}_{et}\left(
A_{f}, \mathbb{Z}(n)
\right)_{l}
\end{equation*}
for $n+1\neq 1$. Therefore the statement follows.

\end{proof}
\begin{rem}\upshape
The sequence (\ref{gp}) is exact and the first map in the sequence 
(\ref{gp})
is an injective in the case where $A$ is an equi-characteristic
regular local ring (even when $\operatorname{dim}(A)\neq 2$) 
by \cite[Theorem C]{P}.
\end{rem}
\begin{thm}\upshape
\label{DWE}

Let $B$ 
be a discrete valuation ring which is an essentially of
finite type over a field.
Let $A$ be a regular local ring which is smooth over $A$ with 
$\operatorname{dim}(A)=2$.

Then the sequence
\begin{align*}
0\to
&\operatorname{H}^{1}(A_{et}, \nu_{r}^{n})\to
\operatorname{Ker}\Bigl(
\operatorname{H}^{1}(k(A)_{et}, \nu_{r}^{n})\to
\prod_{\mathfrak{p}\in \spec A^{(1)}}\operatorname{H}^{1}\left(
k(A_{\bar{\mathfrak{p}}})_{et}, \nu_{r}^{n}
\right)
\Bigr)\\
\to &
\bigoplus_{\mathfrak{p}\in \spec A^{(1)}}
\operatorname{H}^{1}(\kappa(\mathfrak{p})_{et}, \nu_{r}^{n-1})
\end{align*}
is exact. 
\end{thm}
\begin{proof}
We have
\begin{equation*}
\operatorname{H}^{n+1}_{B}(A)_{p^{r}}
=
\operatorname{H}^{1}\left(
A_{et}, \operatorname{W\Omega}^{n}_{r, \log}
\right)
\end{equation*}
by Lemma \ref{LDW}.
Let $i: \spec A/\mathfrak{m}\to \spec A$ be the closed point with the open complement
$j: U\to X$. Then we have 
\begin{equation*}
\operatorname{H}^{n+2}_{\spec (A/\mathfrak{m})}\left(
(\spec A)_{et}, \mathbb{Z}/p^{r}(n)
\right)
=\operatorname{K}_{M}^{n-2}(A/\mathfrak{m})/
p^{r}\operatorname{K}_{M}^{n-2}(A/\mathfrak{m})
\end{equation*}
by the theorem of Gros-Shiho-Suwa \cite[p.583, Theorem 3.2]{Sh}, the theorem of
Bloch-Gabber-Kato \cite[p.117, Corollary (2.8)]{B-K} and Proposition \ref{DWE}. 
Here $\operatorname{K}_{M}^{n-2}(A/\mathfrak{m})$ is 
the Milnor $K$-group of $A/\mathfrak{m}$.
Moreover we have
\begin{equation*}
\operatorname{H}^{n+2}_{\spec (A/\mathfrak{m})}\left(
(\spec A)_{et}, \mathbb{Z}(n)
\right)
=\operatorname{K}_{M}^{n-2}\left(
A/\mathfrak{m}
\right)
\end{equation*}
by \cite{N-Su} and the quasi-isomorphism (\ref{quasiil}). Hence we have
\begin{equation*}
\operatorname{H}^{n+3}_{\spec(A/\mathfrak{m})}\left(
(\spec A)_{et}, \mathbb{Z}(n)
\right)_{p^{r}}
=0
\end{equation*}
by a distinguished triangle
\begin{equation*}
\cdots\to
i_{*}\mathbf{R}i^{!}\mathbb{Z}(n)\xrightarrow{\times p^{r}}
i_{*}\mathbf{R}i^{!}\mathbb{Z}(n)\to
i_{*}\mathbf{R}i^{!}\mathbb{Z}/p^{r}(n)
\to\cdots.
\end{equation*}
Therefore the morphism
\begin{equation*}
\operatorname{H}^{n+2}_{et}\left(A, \mathbb{Z}(n)\right)_{p^{r}}
\to
\operatorname{H}^{n+2}_{et}\left(A, \mathbf{R}j_{*}\mathbb{Z}(n)\right)_{p^{r}}
=\operatorname{H}^{n+2}_{et}\left(
U, \mathbb{Z}(n)\right)_{p^{r}}
\end{equation*}
is an isomorphism. The statement follows from Proposition \ref{supp}.
\end{proof}
\begin{prop}\upshape
\label{sur}
Let $\mathcal{O}_{K}$ be a Henselian discrete valuation ring with residue field $F$ and the quotient
field $K$. 

Let $\mathfrak{X}$ be a connected  regular scheme and 
proper over 
$\spec\mathcal{O}_{K}$ with $\operatorname{dim}(\mathfrak{X})=2$, $X=\mathfrak{X}\otimes_{\mathcal{O}_{K}} K$ 
and $Y=\mathfrak{X}\otimes_{\mathcal{O}_{K}} F$. 

Then the morphism
\begin{equation}\label{mil}
\operatorname{H}^{n+1}_{Zar}\left(X, \mathbb{Z}(n)\right)\to
\operatorname{H}^{n}_{Zar}\left(Y, \mathbb{Z}(n-1)\right)
\end{equation}
is a surjective.
\end{prop}
\begin{proof}
Let
\begin{equation*}
F^{s}z^{n}(X, *)=
\operatorname*{colim}_{\substack{Z\subset X\\ \operatorname{codim}_{X}Z\geq s}}\operatorname{im}\left(
z^{n-s}(Z, *)\to z^{n}(X, *)
\right),
\end{equation*}
be the subcomplex of $z^{n}(X, *)$ generated by cycles whose projection to $X$ has codimension at least
$s$. 
Taking the colimit of the localization sequence (\cite[p.779, Theorem 3.2 (a)]{Ge}) over pairs $Z^{\prime}\subset Z$ 
with $Z$ of 
codimension $s$ and
$Z^{\prime}$ of codimension $s+1$, we get quasi-isomorphisms 
\begin{align*}
gr^{s}z^{n}(X, *)=F^{s}z^{n}(X, *)/F^{s+1}z^{n}(X, *)&\simeq \operatorname*{colim}_{(Z, Z^{\prime})}
z^{n-s}(Z-Z^{\prime}, *)\\
&\xrightarrow[\operatorname{res}]{\simeq}\bigoplus_{x\in X^{(s)}}z^{n-s}(\kappa(x), *).
\end{align*}
Here the right hand map send a cycle on a subscheme to its restriction 
to the generic points (\cite[p.782]{Ge}). Then
\begin{align}\label{momil}
&\operatorname{H}_{Zar}^{n+1}(X, \mathbb{Z}(n)) \nonumber
\\
=&\operatorname{Coker}\left(
\bigoplus_{x\in X^{(0)}}\operatorname{H}^{n}_{Zar}(\kappa(x), \mathbb{Z}(n))
\to \bigoplus_{x\in X^{(1)}}\operatorname{H}^{n-1}_{Zar}(\kappa(x), \mathbb{Z}(n-1))
\right)
\end{align}
by the exact sequence
\begin{equation*}
0\to F^{1}z^{n}(X, *)\to z^{n}(X, *)\to gr^{0}z^{n}(X, *)\to 0.
\end{equation*}
Moreover
\begin{equation}
\operatorname{Coker}\left(
F^{1}z^{n}(\mathfrak{X}, *)\to F^{1}z^{n}(X, *)
\right)\simeq
\operatorname{Coker}\left(
z^{n}(\mathfrak{X}, *)\to z^{n}(X, *)
\right)
\end{equation}
by applying snake lemma to a commutative diagram with two exact rows
\begin{equation*}
\begin{CD}
@. F^{1}z^{n}(\mathfrak{X}, *)@>>> z^{n}(\mathfrak{X}, *)
@>>>\displaystyle\bigoplus_{x\in \mathfrak{X}^{0}}z^{n}(\kappa(x), *)@>>> 0 \\
@. @VVV @VVV @V\simeq VV \\
0@>>>F^{1}z^{n}(X, *)@>>> z^{n}(X, *)
@>>>\displaystyle\bigoplus_{x\in X^{0}}z^{n}(\kappa(x), *).
\end{CD} 
\end{equation*}
Hence we have a distinguish triangle 
\begin{equation}
0\to z^{n-1}(Z, *)\to F^{1}z^{n}(\mathfrak{X}, *)\to F^{1}z^{n}(X, *)
\end{equation}
by \cite[Theorem 1.7]{L1}. Let $x\in X^{(1)}$ and $Z=\Bar{\{x}\}$. Let $\tilde{Z}\to Z$ be the normalization of $Z$. Then
we have the map of distinguished triangles
\begin{equation*}
\begin{CD}
z^{n-2}\left(\tilde{Z}\times_{\mathcal{O}_{k}} k, *\right)@>>>z^{n-1}(\tilde{Z}, *)
@>>>z^{n-1}\left(\tilde{Z}\times_{\mathcal{O}_{k}}K, *\right) \\
@VVV @VVV @VVV \\
z^{n-1}(Y, *) @>>> F^{1}z^{n}(\mathfrak{X}, *)@>>> F^{1}z^{n}(X, *)
\end{CD}
\end{equation*}
by \cite[Proposition 1.7]{F}. Therefore the morphism (\ref{mil}) is induced by the morphism
\begin{math}
\displaystyle\bigoplus_{x\in X^{(1)}}\operatorname{K}^{n}_{M}(\kappa(x))\to
\displaystyle\bigoplus_{x\in Y^{(1)}}\operatorname{K}^{n-1}_{M}(\kappa(x))
\end{math}
which is composite of the symbol map and the norm map by \cite[Lemma 3.2]{Ge-L} and \cite[Lemma 3.4.4]{Su-V}. 

For a point $y\in Y_{0}$, there exists $x\in X_{0}$ such that the residue field $\kappa(x)$ is 
the quotient field of a Henselian discrete valuation ring
\begin{math}
\bar{\{x\}}\subset \mathfrak{X}
\end{math}
with the residue field $\kappa(y)$ (c.f, \cite[p.262, The proof of Lemma (3.11) (2)]{Sa}). Therefore the
statement follows.
\end{proof}
\begin{cor}\upshape\label{vm}
Let notations be same as Proposition \ref{sur}. Then 
\begin{equation*}
\operatorname{H}^{n+2}_{Zar}\left(\mathfrak{X}, \mathbb{Z}(n)\right)=0. 
\end{equation*}
\end{cor}
\begin{proof}
The statement follows from Proposition \ref{sur} and \cite[p.780, Corollary 3.3 a)]{Ge}. 
\end{proof}
\begin{thm}\upshape
\label{CDVRS}
Let $\mathfrak{X}$ be same as Proposition \ref{sur}.
Then Problem \ref{prob} (iii) for $\mathfrak{X}$ is true.
\end{thm}
\begin{proof}
Since 
\begin{math}
\mathcal{H}^{n+2}\left(
\mathbb{Z}(n)^{\mathfrak{X}}_{et}
\right)=0,
\end{math}
we have a distinguished triangle
\begin{equation*}
\cdots\to
\mathbb{Z}(n)^{\mathfrak{X}}_{et}\to
\tau_{\leq n+1}\mathbf{R}j_{*}\mathbb{Z}(n)_{et}\to
i_{*}\left(
\tau_{\leq n+1}\left(
(\mathbf{R}i^{!}\mathbb{Z}(n)_{et})[+1]
\right)\right)
\to\cdots
\end{equation*}
and a quasi-isomorphism
\begin{align*}
&\tau_{\leq n+1}\left(
(\mathbf{R}i^{!}\mathbb{Z}(n)_{et})[+1]
\right)
=
\left(
\tau_{\leq n+2}
\mathbf{R}i^{!}
\mathbb{Z}(n)_{et}
\right)[+1] \\
\simeq
&\left(\tau_{\leq n+2}\left(
\mathbb{Z}(n-1)_{et}[-2]
\right)\right)[+1]
=
\tau_{\leq n+1}\left(
\mathbb{Z}(n-1)_{et}[-1]
\right)
\end{align*}
by Proposition \ref{tru2}. Since the homomorphism
\begin{align*}
&\operatorname{H}^{n+2}_{et}\left(
\mathfrak{X},
i_{*}\left(
\tau_{\leq n+1}\left(
(\mathbf{R}i^{!}\mathbb{Z}(n)_{et})[+1]
\right)
\right)
\right)\\
=
&\operatorname{H}^{n+2}_{et}\left(
Y,
\tau_{\leq n+1}\left(
(\mathbf{R}i^{!}\mathbb{Z}(n)_{et})[+1]
\right)
\right)\\
=
&\operatorname{H}^{n+2}_{et}\left(
Y,
\tau_{\leq n+1}\left(
\mathbb{Z}(n-1)_{et}[-1]
\right)
\right)\\
\hookrightarrow
&\operatorname{H}^{n+2}_{et}\left(
Y, \mathbb{Z}(n-1)_{et}[-1] 
\right)
=
\operatorname{H}^{n+1}_{et}
\left(Y,
\mathbb{Z}(n-1)
\right)
\end{align*}
is an injective by Lemma \ref{tru}, we have an exact sequence
\begin{equation}
\label{exactmain}
\operatorname{H}_{et}^{n+2}\left(
\mathfrak{X}, \mathbb{Z}(n)
\right)\to
\operatorname{H}_{et}^{n+2}\left(
\mathfrak{X}, \tau_{\leq n+1}\mathbf{R}j_{*}\mathbb{Z}(n)
\right)\to
\operatorname{H}_{et}^{n+1}\left(
Y, \mathbb{Z}(n-1)
\right)
\end{equation}
and the first map of the sequence (\ref{exactmain}) is an injective by Proposition 
\ref{sur}.
Moreover we have
\begin{align*}
&\operatorname{H}_{et}^{n+2}\left(
\mathfrak{X}, \tau_{\leq n+1}\mathbf{R}j_{*}\mathbb{Z}(n)
\right) \\
=
&\operatorname{Ker}\Bigl(
\operatorname{H}_{et}^{n+2}\left(
X, \mathbb{Z}(n)
\right)\to
\prod_{x\in \mathfrak{X}_{(0)}}
\operatorname{H}_{et}^{n+2}\left(
k(\mathcal{O}_{\mathfrak{X}, \bar{x}}), \mathbb{Z}(n)
\right)
\Bigr)
\end{align*}

and the homomorphism
\begin{equation*}
\operatorname{H}_{et}^{n+1}\left(
Y, \mathbb{Z}(n-1)
\right)\to
\bigoplus_{x\in Y^{(0)}}
\operatorname{H}_{et}^{n+1}\left(
\kappa(x), \mathbb{Z}(n-1)
\right)
\end{equation*}
is an injective by a resolution of singularities
of curves and Proposition \ref{supp}. Moreover
\begin{math}
Y^{(0)}\subset \mathfrak{X}^{(1)}. 
\end{math}
Therefore the statement follows from Proposition \ref{supp}.

\end{proof}
\begin{cor}\upshape
\label{GSS}
Let notations be same as above. Then we have
\begin{equation}\label{int}
\operatorname{H}_{B}^{n}(\mathfrak{X})=\bigcap_{x\in \mathfrak{X}_{(0)}}
\operatorname{H}_{B}^{n}(\mathcal{O}_{\mathfrak{X}, x}).
\end{equation}
Suppose that $\mathcal{O}_{K}$ is a discrete valuation ring which is
an essentially of finite type over a field. 
Suppose that $\mathfrak{X}$ is smooth over $\spec\mathcal{O}_{K}$. 
Then the sequence
\begin{equation}\label{GSsur}
0 \to \operatorname{H}_{B}^{n}(\mathfrak{X}) \to
\operatorname{Ker}\Bigl(
\operatorname{H}_{B}^{n}(k(\mathfrak{X}))\to
\prod_{x\in \mathfrak{X}^{(1)}}
\operatorname{H}_{B}^{n}(k(\mathcal{O}_{\mathfrak{X}, \bar{x}}))
\Bigr)\to
\bigoplus_{x\in \mathfrak{X}^{(1)}}
\operatorname{H}_{B}^{n-1}(\kappa(x))
\end{equation}
is exact.
\end{cor}
\begin{proof}
The equation (\ref{int}) follows from Proposition \ref{rl2} and Theorem \ref{CDVRS}. 
If $\mathfrak{X}$ satisfies the assumption of Corollary \ref{GSS},
then the exact sequence (\ref{GSsur})
follows from (\ref{int}), \cite[Theorem C]{P} and Theorem \ref{DWE}. 
\end{proof}
\section{Logarithmic Hodge-Witt cohomology of Henselian regular local rings}
The objective of this section is to prove Theorem \ref{hendR}. 
Theorem \ref{hendR} is used to prove Proposition \ref{GBsur}.
\begin{lem}\upshape\label{dRex}
Let $A$ be a regular local ring over $\mathbb{F}_{p}$.
Let $t$ be a regular element of $A$, $\mathfrak{n}=(t)$ and
$B=A/\mathfrak{n}$. Then we have an exact sequence
\begin{equation*}
0
\to
(
\mathfrak{n}\Omega^{i}_{A}+d\Omega^{i-1}_{A}
)
/d\Omega^{i-1}_{A}
\to
\Omega^{i}_{A}/d\Omega^{i-1}_{A}
\xrightarrow{\bar{g}_{i}}
\Omega^{i}_{B}/d\Omega^{i-1}_{B}
\to
0
\end{equation*}
where the homomorphism $\bar{g}_{i}$ is induced by the natural homomorphism
\begin{math}
g_{i}: \Omega^{i}_{A}
\xrightarrow{}
\Omega^{i}_{B} 
. 
\end{math}
\end{lem}
\begin{proof}\upshape
$A$ can be written as a filtering inductive limit
\begin{math}
\displaystyle\lim_{\substack{\to\\ \lambda}}A_{\lambda} 
\end{math}
of finitely generated smooth algebras over $\mathbb{F}_{p}$
by Popescu's theorem (\cite{Po}). Let 
$\mathfrak{n}_{\lambda}$ be an ideal $(t)$ of $A_{\lambda}$ and 
$B_{\lambda}=A_{\lambda}/(t)$. 
Then we may assume that $B_{\lambda}$ is $0$-smooth over 
$\mathbb{F}_{p}$.

By
\cite[p.194, Theorem 25.2]{Ma}, 
we have a split exact sequence
\begin{equation*}
0
\to
\mathfrak{n}_{\lambda}/\mathfrak{n}_{\lambda}^{2}
\xrightarrow{\delta_{0}}
\Omega_{A_{\lambda}}\otimes_{A_{\lambda}} B_{\lambda}
\xrightarrow{\alpha_{0}}
\Omega_{B_{\lambda}}
\to 
0
\end{equation*}
where a section $\gamma_{0}$ of $\alpha_{0}$ is defined by 
\begin{equation*}
\gamma_{0}\left(
\bar{b}d\bar{a}
\right)
=
da\otimes \bar{b}
\end{equation*}
for $a, b\in A$.
Therefore we have an exact sequence
\begin{equation*}
0
\to
\mathfrak{n}_{\lambda}/\mathfrak{n}_{\lambda}^{2}
\otimes_{B_{\lambda}}\Omega^{i-1}_{B_{\lambda}}
\xrightarrow{\delta_{i}}
\Omega^{i}_{A_{\lambda}}\otimes_{A_{\lambda}} B_{\lambda}
\xrightarrow{\alpha_{i}}
\Omega^{i}_{B_{\lambda}}
\to 
0
\end{equation*}
where
\begin{equation*}
\delta_{i} 
\left(
\bar{a}\otimes
(
d\bar{b}_{1}
\wedge\cdots\wedge
d\bar{b}_{i-1}
)
\right)
=
\left(
da
\wedge
db_{1}
\wedge\cdots\wedge
db_{i-1}
\right)
\otimes 
\bar{1}
\end{equation*}
for $a\in \mathfrak{n}_{\lambda}$, 
$b_{1},\cdots, b_{i-1}\in A_{\lambda}$
and
\begin{equation*}
\alpha_{i}
\left(
(
dc_{1}
\wedge\cdots\wedge
dc_{i}
)
\otimes
\bar{f}
\right)
=
\bar{f}
d\bar{c}_{1}
\wedge\cdots\wedge
d\bar{c}_{i}
\end{equation*}
for
\begin{math}
c_{1}, \cdots, c_{i}, f\in A_{\lambda}
\end{math}
by \cite[p.284, Theorem C.2]{Ma}.
Hence the sequence
\begin{equation}\label{deRhamex}
0
\to
\mathfrak{n}/\mathfrak{n}^{2}
\otimes_{B}\Omega^{i-1}_{B}
\xrightarrow{\delta_{i}}
\Omega^{i}_{A}\otimes_{A} B
\xrightarrow{\alpha_{i}}
\Omega^{i}_{B}
\to 
0
\end{equation}
is exact.

On the other hand, we have
\begin{align*}
\operatorname{Ker}\left(
\alpha_{i}
\right)
=
\operatorname{Im}\left(
\delta_{i}
\right)
\subset
\left(
d\Omega^{i-1}_{A}+\mathfrak{n}\Omega^{i}_{A}
\right)
/
\mathfrak{n}\Omega^{i}_{A}
\end{align*}
by the sequence (\ref{deRhamex}). So we have
\begin{equation*}
\operatorname{Ker}\left(
g_{i}
\right)
\subset
d\Omega^{i-1}_{A}+\mathfrak{n}\Omega^{i}_{A}.
\end{equation*}
Since 
\begin{equation*}
d\Omega^{i-1}_{B}
=
\operatorname{Im}
\left(
d\Omega^{i-1}_{A}
\to
\Omega^{i}_{A}
\xrightarrow{g_{i}}
\Omega^{i}_{B} 
\right),
\end{equation*}
we have
\begin{equation*}
\operatorname{Ker}\left(
\bar{g}_{i}
\right)
\subset
\left(
d\Omega^{i-1}_{A}+\mathfrak{n}\Omega^{i}_{A}
\right)
/d\Omega^{i-1}_{A}.
\end{equation*}
Therefore the statement follows.
\end{proof}
\begin{lem}\upshape\label{homV}
Let $A$ be a Henselian regular local ring 
and
$X=\spec(A)$. Suppose that $\operatorname{char}(A)=p>0$. 
Then
\begin{equation*}
\operatorname{H}^{1}
\left(
X_{et}, Z\Omega^{i}_{X}
\right)
=
\operatorname{H}^{1}
\left(
X_{et}, d\Omega^{i}_{X}
\right)
=
0
\end{equation*}
for all $i>0$.
\end{lem}
\begin{proof}\upshape
Since 
$\Omega^{i}_{X}$
is a quasi-coherent sheaf of $\mathcal{O}_{X}$-module,
\begin{equation*}
\operatorname{H}^{j}(
X_{Zar},
\Omega^{i}_{X}
)
=
0
\end{equation*}
for all $j>0$ by \cite[p.103, III, Lemma 2.15]{M}. Hence
we have
\begin{equation*}
\operatorname{H}^{j}(
X_{et},
\Omega^{i}_{X}
)
=
0 
\end{equation*}
by \cite[p.48, II, Proposition 1.3]{M} and
\cite[p.114, III, Remark 3.8]{M}. 
Moreover, since $\operatorname{cd}_{p}(k_{et})=1$,
we have
\begin{equation*}
\operatorname{H}^{2}
\left(
X_{et}, \Omega^{i}_{X, \log}
\right)
=
\operatorname{H}^{2}
\left(
k_{et}, i^{*}(
\Omega^{i}_{X, \log}
)
\right)
=0
\end{equation*}
by 
\cite[p.777, the proof of Proposition 2.2.b)]{Ge}.
Therefore we have
\begin{equation*}
\operatorname{H}^{2}\left(
X_{et},
Z\Omega^{i}_{X}
\right)
=
0 
\end{equation*}
by \cite[p.576, Proposition 2.10]{Sh}
and
\begin{equation*}
\operatorname{H}^{1}(
X_{et},
d\Omega^{i}_{X}
)
=
0 
\end{equation*}
by the exact sequence
\begin{equation*}
0
\to
Z\Omega^{i}_{X}
\to
\Omega^{i}_{X}
\to
d\Omega^{i}_{X}
\to 
0.
\end{equation*}
Since the sequence
\begin{equation*}
0
\to
d\Omega^{i-1}_{X}
\to
Z\Omega^{i}_{X}
\to
\Omega^{i}_{X}
\to
0
\end{equation*}
is exact by \cite[p.574, Proposition 2.5]{Sh},
we have
\begin{equation*}
\operatorname{H}^{1}\left(
X_{et}, Z\Omega^{i}_{X}
\right)
=
0.
\end{equation*}
Therefore the statement follows.
\end{proof}
\begin{thm}\upshape
\label{hendR}
Let $A$ be a Henselian regular local ring
over $\mathbb{F}_{p}$, 
$t$ a regular element of $A$ and $B=A/(t)$. 

Then the homomorphism
\begin{equation}\label{pos}
\operatorname{H}^{1}\left(
(\spec A)_{et}, \Omega^{i}_{A, \log}
\right)
\to
\operatorname{H}^{1}\left(
(\spec B)_{et}, \Omega^{i}_{B, \log}
\right)
\end{equation}
is an isomorphism.
\end{thm}
\begin{proof}\upshape
We have the following commutative diagram.
\begin{equation}\label{dR}
\begin{CD}
\operatorname{Ker}(g_{i})
@.\to
(
\mathfrak{n}\Omega^{i}_{A}+d\Omega^{i-1}_{A}
)
/d\Omega^{i-1}_{A}
\\
@VVV @VVV\\
\Omega^{i}_{A}
@.\xrightarrow{~~~1-F~~~}
\Omega^{i}_{A}/d\Omega^{i-1}_{A}
@.\xrightarrow{~~~}
\operatorname{H}^{1}\left(
(\spec A)_{et}, \Omega^{i}_{A, \log}
\right)
@.\to
0\\
@V{g_{i}}VV @VVV @VVV \\
\Omega^{i}_{B}
@.\xrightarrow{~~~1-F~~~}
\Omega^{i}_{B}/d\Omega^{i-1}_{B}
@.\xrightarrow{~~~}
\operatorname{H}^{1}\left(
(\spec B)_{et}, \Omega^{i}_{B, \log}
\right)    \\
@VVV \\
0
\end{CD}
\end{equation}
where $F$ is the homomorphism which is induced by the Frobenius operator. 

Then the horizontal arrows in (\ref{dR}) are exact by 
\cite[p.576, Proposition 2.8]{Sh} and Lemma \ref{homV}.
Moreover the vertical arrow in (\ref{dR}) is exact by Lemma \ref{dRex}. 

If the upper homomorphism in (\ref{dR})
is a surjective, we can show that
the homomorphism (\ref{pos}) is an injective
by chasing diagram (\ref{dR}). We have a surjective homomorphism
\begin{align*}
A\otimes(A^{*})^{\otimes i}
&\to
\Omega^{i}_{A}  \\
a\otimes b_{1}\otimes \cdots \otimes b_{i}
&\mapsto
a\frac{db_{1}}{b_{1}}\wedge\cdots\wedge\frac{db_{i}}{b_{i}}
\end{align*}
by \cite[p.34, Lemma]{K-F} and
\begin{equation*}
F\left(
a\frac{db_{1}}{b_{1}}\wedge\cdots\wedge\frac{db_{i}}{b_{i}}
\right)
=
a^{p}\frac{db_{1}}{b_{1}}\wedge\cdots\wedge\frac{db_{i}}{b_{i}}.
\end{equation*}
Since
\begin{equation*}
\mathfrak{n}\Omega^{i}_{A}
\subset
\operatorname{Ker}(g_{i}),
\end{equation*}
it is sufficient to show that for any $a\in \mathfrak{n}$
there exists a $b\in \mathfrak{n}$ such that
\begin{equation}\label{sol}
b^{p}-b=a. 
\end{equation}
By the definition of Henselian,
there exists a $b\in A\setminus A^{*}$ such that $b$ is a solution of the 
equation (\ref{sol}) and 
\begin{math}
b+1,\cdots, b+p-1\in A^{*} 
\end{math}
are also solutions of the equation (\ref{sol}).
Hence $b\in \mathfrak{n}$ by the equation (\ref{sol}). 
The homomorphism (\ref{pos}) is also a surjective
by the diagram (\ref{dR}). Therefore the statement follows.
\end{proof}
\section{A generalization of Artin's theorem}

In this section, we consider a generalization of Artin's theorem 
(\cite[p.98, Theorem (3.1)]{Gr}) as follows.
\begin{prop}\upshape
\label{GBsur}
Let $\mathcal{O}_{K}$ be the Henselization of a discrete valuation ring
which is an essentially of finite type over a field $k$. 
Let 
\begin{math}
\pi: \mathfrak{X}\to \spec \mathcal{O}_{K}
\end{math}
be a smooth and proper morphism and
$Y$ the closed fiber of $\pi$. Suppose that $\mathfrak{X}$ is a connected and 
$\operatorname{dim}(\mathfrak{X})=2$.

Then we have the morphism
\begin{equation}\label{ci}
\operatorname{H}^{n}_{\operatorname{B}}\left(
\mathfrak{X}
\right)
\to
\operatorname{H}^{n}_{\operatorname{B}}\left(
Y
\right)
\end{equation}
and the morphism (\ref{ci}) is a surjective.

\end{prop}
\begin{proof}
 
The morphism (\ref{ci}) is defined by \cite[Remark 2.9]{L2} and \cite[Theorem 10.2]{L2}. 
Let
\begin{align*}
X_{et}
\xrightarrow{\epsilon}
X_{Zar},
&&
X_{et}
\xrightarrow{\alpha}
X_{Nis}
\xrightarrow{\beta}
X_{Zar}
\end{align*}
be the canonical maps of sites. Since
\begin{equation*}
\mathbb{Z}/m(n)_{Zar}
\xrightarrow{\sim}
\tau_{\leq n}
\left(
\mathbf{R}\epsilon_{*}\mathbb{Z}/m(n)_{et}
\right)
\end{equation*}
by \cite[p.774, Theorem 1.2.2]{Ge} and 
$\beta^{*}\beta_{*}=\operatorname{id}$,
we have a quasi-isomorphism
\begin{equation*}
\mathbb{Z}/m(n)_{Nis} 
\xrightarrow{\sim}
\tau_{\leq n}
\left(
\mathbf{R}\alpha_{*}\mathbb{Z}/m(n)_{et}
\right)
\end{equation*}
for any positive integer $m$.
Hence we have a distiguished triangle
\begin{equation}\label{dist}
\mathbb{Z}/m(n)_{Nis}
\to
\tau_{\leq n+1}
\left(
\mathbf{R}\alpha_{*}\mathbb{Z}/m(n)_{et}
\right)
\to
\mathbf{R}^{n+1}\alpha_{*}\mathbb{Z}/m(n)_{et}[-n-1].
\end{equation}
Since 
\begin{equation*}
\operatorname{H}^{n+2}_{Zar}\left(
\mathfrak{X}, \mathbb{Z}(n)
\right)
=
\operatorname{H}^{n+2}_{Zar}\left(
Y, \mathbb{Z}(n)
\right)
=
0
\end{equation*}
by Corollary \ref{vm} and Vanishing theorem,
we have
\begin{equation}\label{NisV}
\operatorname{H}^{n+2}_{Nis}\left(
\mathfrak{X}, \mathbb{Z}(n)
\right)
=
\operatorname{H}^{n+2}_{Nis}\left(
Y, \mathbb{Z}(n)
\right)
=
0
\end{equation}
by \cite[p.781, Proposition 3.6]{Ge} and
\begin{equation}\label{NismV}
\operatorname{H}^{n+2}_{Nis}\left(
\mathfrak{X}, \mathbb{Z}/m(n)
\right)
=
\operatorname{H}^{n+2}_{Nis}\left(
Y, \mathbb{Z}/m(n)
\right)
=
0
\end{equation}
by the distinguished triangle 
\begin{equation}\label{distm}
\mathbb{Z}(n)_{t}
\xrightarrow{\times m}
\mathbb{Z}(n)_{t}
\to
\mathbb{Z}/m(n)_{t}
\end{equation}
for 
$t\in\{Zar, Nis, et\}$.
Hence we have a commutative diagram
\begin{equation}\label{Gam}
\begin{CD}
0
\to
\operatorname{H}^{n+1}_{Nis}\left(
\mathfrak{X}, \mathbb{Z}/m(n)
\right)
@. \to
\operatorname{H}^{n+1}_{et}\left(
\mathfrak{X}, \mathbb{Z}/m(n)
\right)
@. \to
\Gamma\left(
\mathfrak{X}, \mathbf{R}^{n+1}\alpha_{*}\mathbb{Z}/m(n)_{et}
\right)
\to
0 \\
@VVV @VVV @VVV \\
0
\to
\operatorname{H}^{n+1}_{Nis}\left(
Y, \mathbb{Z}/m(n)
\right)
@. \to
\operatorname{H}^{n+1}_{et}\left(
Y, \mathbb{Z}/m(n)
\right)
@. \to
\Gamma\left(
Y, \mathbf{R}^{n+1}\alpha_{*}\mathbb{Z}/m(n)_{et}
\right)
\to
0 
\end{CD}
\end{equation}
with exact rows by the distinguished triangle (\ref{dist}).

On the other hand, 
\begin{align*}
&\operatorname{Im}\left(
\operatorname{H}^{n+1}_{et}\left(
\mathfrak{X}, \mathbb{Z}(n)
\right)
\to
\operatorname{H}^{n+1}_{et}\left(
\mathfrak{X}, \mathbb{Z}/m(n)
\right)
\right) \\
=&
\operatorname{Coker}\left(
\operatorname{H}^{n+1}_{et}\left(
\mathfrak{X}, \mathbb{Z}(n)
\right)
\to
\operatorname{H}^{n+1}_{et}\left(
\mathfrak{X}, \mathbb{Z}(n)
\right)
\right)\\
=&
\operatorname{H}^{n+1}_{Nis}\left(
\mathfrak{X}, \mathbb{Z}/m(n)
\right)
\end{align*}
and
\begin{align*}
&\operatorname{Im}\left(
\operatorname{H}^{n+1}_{et}\left(
Y, \mathbb{Z}(n)
\right)
\to
\operatorname{H}^{n+1}_{et}\left(
Y, \mathbb{Z}/m(n)
\right)
\right) 
=
\operatorname{H}^{n+1}_{Nis}\left(
Y, \mathbb{Z}/m(n)
\right)
\end{align*}
by the distinguished triangle (\ref{distm}) and (\ref{NisV}).

So we have a commutative diagram
\begin{equation}\label{etm}
\begin{CD}
0
\to
\operatorname{H}^{n+1}_{Nis}\left(
\mathfrak{X}, \mathbb{Z}/m(n)
\right)
@. \to
\operatorname{H}^{n+1}_{et}\left(
\mathfrak{X}, \mathbb{Z}/m(n)
\right)
@. \to
\operatorname{H}^{n+2}_{et}\left(
\mathfrak{X}, \mathbb{Z}(n)
\right)_{m}
\to
0 \\
@VVV @VVV @VVV \\
0
\to
\operatorname{H}^{n+1}_{Nis}\left(
Y, \mathbb{Z}/m(n)
\right)
@. \to
\operatorname{H}^{n+1}_{et}\left(
Y, \mathbb{Z}/m(n)
\right)
@. \to
\operatorname{H}^{n+2}_{et}\left(
Y, \mathbb{Z}(n)_{et}
\right)_{m}
\to
0 
\end{CD}
\end{equation}
with exact rows by the distinguished triangle (\ref{distm}).

Therefore it is sufficient to show that the homomorphism
\begin{equation}\label{Nisi}
\Gamma\left(
\mathfrak{X},\mathbf{R}^{n+1}\alpha_{*}\mathbb{Z}/m(n)_{et}
\right)
\to
\Gamma\left(
Y,\mathbf{R}^{n+1}\alpha_{*}\mathbb{Z}/m(n)_{et}
\right)
\end{equation}
is a surjective by the commutative diagrams (\ref{Gam}) and (\ref{etm}).

Let $i: Y\to \mathfrak{X}$ be the closed immersion of $\mathfrak{X}$. Then the diagram
\begin{equation}\label{Nisdia}
\begin{CD}
\operatorname{H}^{n+1}_{et}\left(
\mathfrak{X}, \mathbb{Z}/m(n)
\right)
@. \to
\Gamma\left(
\mathfrak{X}, \mathbf{R}^{n+1}\alpha_{*}\mathbb{Z}/m(n)_{et}
\right)
\to
0 \\
@VVV @VVV  \\
\operatorname{H}^{n+1}_{et}\left(
Y, i^{*}\mathbb{Z}/m(n)
\right)
@. \to
\Gamma\left(
Y, \mathbf{R}^{n+1}\alpha_{*}i^{*}\mathbb{Z}/m(n)_{et}
\right)
\\
@VVV @VVV  \\
\operatorname{H}^{n+1}_{et}\left(
Y, \mathbb{Z}/m(n)
\right)
@. \to
\Gamma\left(
Y, \mathbf{R}^{n+1}\alpha_{*}\mathbb{Z}/m(n)_{et}
\right)
\to
0 
\end{CD}
\end{equation}
is commutative.

Suppose that $m$ is prime to $\operatorname{char}(\mathfrak{X})$.
Then we have
\begin{equation*}
i^{*}\left(
\mathbb{Z}/m(n)^{\mathfrak{X}}
\right)
=\mathbb{Z}/m(n)^{Y}
\end{equation*}
by the definition of $m$-th roots of unity.
So the homomorphism (\ref{Nisi}) is a surjective by the proper 
base change theorem (\cite[p.224, Corollary 2.7]{M}).


Suppose that $m=\operatorname{char}(\mathfrak{X})$. Then 
\begin{math}
\operatorname{H}^{n+2}_{Nis}\left(
Y, i^{*}\mathbb{Z}/m(n)
\right)
=0
\end{math}
by \cite{K-S}.
So the middle horizontal 
homomorphism in the diagram (\ref{Nisdia}) is a surjective. Therefore the homomorphism
\begin{equation*}
\Gamma\left(
\mathfrak{X},\mathbf{R}^{n+1}\alpha_{*}\mathbb{Z}/m(n)_{et}
\right)
\to
\Gamma\left(
Y,\mathbf{R}^{n+1}\alpha_{*}i^{*}\mathbb{Z}/m(n)_{et}
\right)
\end{equation*}
is a surjective by the proper base change theorem (\cite[p.224, Corollary 2.7]{M}). 
Moreover the homomorphism
\begin{equation*}
\Gamma\left(
Y,\mathbf{R}^{n+1}\alpha_{*}i^{*}\mathbb{Z}/m(n)_{et}
\right)
\to
\Gamma\left(
Y,\mathbf{R}^{n+1}\alpha_{*}\mathbb{Z}/m(n)_{et}
\right)
\end{equation*}
is a surjective by \cite[p.777, The proof of Proposition 2.2.b)]{Ge} 
and Theorem \ref{hendR}. 
Therefore the homomorphism (\ref{Nisi}) is a surjective.
Hence the statement follows.


\end{proof}
\begin{rem}\upshape
Let $\mathcal{O}_{K}$ be an excellent Henselian discrete valuation ring.
Let $\mathfrak{X}$ be a regular proper curve over $\spec \mathcal{O}_{K}$.

Suppose that $\mathcal{O}_{K}$ is an equi-dimensional complete discrete valuation and 
$\mathcal{X}$ is projective over $\spec\mathcal{O}_{K}$. 
Then the local-global map
\begin{equation}\label{lg}
\operatorname{H}^{n}_{B}(k(\mathfrak{X}))_{l}\to \prod_{x\in \mathfrak{X}^{(1)}}
\operatorname{H}^{n}_{B}(k(\tilde{\mathcal{O}}_{\mathfrak{X}, x}))_{l}
\end{equation}
is an injective for $n>1$ and 
$l$ which is prime to $\operatorname{char}(\mathcal{O}_{K})$ by \cite[Theorem 3.3.6]{H-H-K}.

Suppose that $\mathcal{O}_{K}$ is an equi-characteristic and 
$\mathfrak{X}$ is a smooth  over $\spec \mathcal{O}_{K}$. Then the local-global map 
(\ref{lg}) is an injective for $n \geq 1$
and any positive integer $l$ if the morphism (\ref{Nisi}) is an injective.

Suppose that $\mathcal{O}_{K}$ is a mixed-characteristic complete discrete
valuation ring and 
$\mathfrak{X}$ is projective over $\spec \mathcal{O}_{K}$. Then the local-global map
(\ref{lg}) is an injective for $n>1$ and $l$ which is prime to
$\operatorname{char}(\mathcal{O}_{K})$ if 
the first map in the exact sequence (\ref{gp}) is an injective 
by \cite[Remark 3.3.7]{H-H-K}.
\end{rem}

\end{document}